\documentclass[11pt,reqno]{amsart}
\usepackage{algorithm,algorithmicx}
\usepackage{algpseudocode}
\usepackage{bbm}
\usepackage[mathlines,pagewise,left]{lineno}
\usepackage{amssymb}
\usepackage{color}
\usepackage{amsmath}
\usepackage{bcol}
\usepackage{multirow}
\usepackage{xspace}
\usepackage{units}
\usepackage[numbers]{natbib}

\usepackage{comment}
\usepackage[toc,page]{appendix}
\usepackage{graphicx} 
\usepackage{booktabs} 
\usepackage{pdflscape}
\usepackage[export]{adjustbox}
\usepackage{subfig}
\usepackage{url}

\usepackage[margin=1in]{geometry}
\usepackage{amsthm}
\newtheorem{proposition}{Proposition}

\newtheorem{theorem}{Theorem}
\newtheorem{corollary}{Corollary}

\theoremstyle{definition}

\theoremstyle{remark}

\newtheorem{example}{Example}

\newcommand{\R}{\ensuremath{\mathbb{R}}}

\newcommand{\F}{\ensuremath{\mathcal{F}}\xspace}
\newcommand{\E}{\ensuremath{\mathbb{E}}}

\newcommand{\be}[1]{\begin{equation}\label{#1}}
\newcommand{\ee}{\end{equation}}

\DeclareMathOperator*{\argmax}{arg\,max}

\newcommand{\hsn}[1]{{\color{blue} (Shaoning: #1)}}
\newcommand{\indicator}[1]{\ensuremath{\mathbbm{1}_{#1}}}

\newcommand{\np}{\ensuremath{\mathcal{NP}}}
\newcommand{\bigO}[1]{\ensuremath{\mathcal{O}(#1 )}}

\usepackage{amsaddr}

\def\DoubleSpacedXI{\linespread{1.5}}

\DoubleSpacedXI
\usepackage[colorinlistoftodos,prependcaption,textsize=small]{todonotes}
\title[]{\large Fractional 0--1 programming and Submodularity}
\author{Shaoning Han$^a$, Andr\'{e}s G\'{o}mez$^a$ and Oleg A. Prokopyev$^{b,*}$\footnote{$^*$Corresponding author. Phone: +1-412-624-9833. E-mail: droleg@pitt.edu.} }
\address{$^a$Industrial \& Systems Engineering, University of Southern California, Los Angeles, CA 90089 \\ $^b$Industrial Engineering, University of Pittsburgh, Pittsburgh, PA 15261}

\begin{document}
\begingroup
\def\uppercasenonmath#1{} 
\let\MakeUppercase\relax 
\maketitle
\endgroup
	
	\begin{abstract}
		 \looseness-1 In this note we study multiple-ratio fractional 0--1 programs, a broad class of \np-hard combinatorial optimization problems. In particular, under some relatively mild assumptions we provide a complete characterization of the conditions, which ensure that a single-ratio function is submodular. Then we illustrate our theoretical results with the assortment optimization and facility location problems, and discuss practical situations that guarantee submodularity in the considered application settings. In such cases, near-optimal solutions for multiple-ratio fractional 0--1 programs can be found via simple greedy algorithms.
		
	\end{abstract}
\section{Introduction}\label{sec:intro}
We consider a multiple-ratio \emph{fractional} 0--1 program given by:
\begin{equation}\label{eq:intro}
	\max_{x\in\mathcal{F}}\;\sum_{k\in M}\frac{\sum_{i\in N}a_{ki}x_i}{b_{k0}+\sum_{i\in N}b_{ki}x_i},
\end{equation}
where $M=\{1,\dots,m\}$, $N=\{1,\dots,n\}$ and $\mathcal{F}:=\{x\in\{0,1\}^n: \ Dx\le d\}$ for given $D\in\R^{q\times n}$ and $d\in\R^q$.  Problem \eqref{eq:intro} is often referred to as a multiple-ratio \emph{hyperbolic} 0--1 program. Problems of the form \eqref{eq:intro} can also be viewed as a class of set-function optimization problems that seek a subset $S$ of $N$ with its indicator variable $\indicator{S}\in\R^n$, where the $i$-th element of  $\indicator{S}$ is 1 if and only if $i\in S$.

Throughout the paper, we make the following assumptions:

\vspace{1mm}

\textbf{A1:} The denominator is strictly positive for each ratio in (\ref{eq:intro}), i.e.,  $b_{k0}+\sum_{i\in N}b_{ki}x_i>0$ for all $k\in M$ and all $x\in \mathcal{F}$.

\textbf{A2:} $a_{ki}\ge0$, $b_{k0}\ge0$ and $b_{ki}>0$ for all $k\in M$ and $i\in N$.

\vspace{1mm}

\looseness-1Assumption \textbf{A1} is standard in fractional optimization \cite{borrero2016simple,borrero2017fractional,mehmanchi2019fractional}. In particular, it ensures that the objective function is well defined. Assumption \textbf{A2} is not too restrictive as it naturally holds in many application settings, see examples in~\cite{borrero2017fractional}, including those considered in this note, see Section~\ref{sec:application}. Finally, for our results developed in this note we also require an additional relatively mild assumption on the structure of the feasible region, $\mathcal{F}$, in \eqref{eq:intro}; it is formalized in Section~\ref{sec:preliminaries}.


Applications of single- and multiple-ratio fractional 0--1 programs as in \eqref{eq:intro} appear in many diverse areas. For example, \citet{mendez2014branch} discuss an assortment optimization problem under mixed multinomial logit choice models (MMNL). \citet{tawarmalani2002global} consider a facility location problem, where a fixed number of facilities need to be located to service customers locations with the objective of maximizing a market share.  \citet{arora1977set} study a class of set covering problems in the context of airline crew scheduling that aim at covering all flights operated by an airline company. Furthermore, many combinatorial optimization problems can be formulated in the form \eqref{eq:intro} including the minimum fractional spanning tree problem \citep{chandrasekaran1977minimal, ursulenko2013global}, the maximum mean-cut problem \citep{iwano1994new,Radzik1998} and the maximum clique ratio problem \citep{sethuraman2015maximum}. More application examples can be found in the studies by \cite{bertsimas2019identifying,elhedhli2005exact,chen2020approaches}, the recent survey by \citet{borrero2017fractional} and the references therein.

Generally speaking, problem \eqref{eq:intro} is \np-hard even in the case of a single ratio~\cite{hansen1991hyperbolic,prokopyev2005complexity}. Moreover, this problem is even hard to approximate, see, e.g., \cite{prokopyev2005complexity}. Also, \citet{rusmevichientong2014assortment} show that for the unconstrained multi-ratio problem, there is no approximation algorithm with polynomial running time that has an approximation factor better than $\bigO{1/m^{1-\delta}}$ for any $\delta>0$. Other related theoretical computational results are discussed in \cite{prokopyev2005complexity,prokopyev2005multiple}.

\looseness-1Exact solution methods for \eqref{eq:intro} encompass mixed-integer programming reformulations \citep{borrero2016simple,elhedhli2005exact,mehmanchi2019fractional}, branch and bound algorithms \citep{tawarmalani2002global}, and other enumerative methods \citep{borrero2017fractional,granot1976solving, hansen1990boolean}. However, due to \np-hardness of \eqref{eq:intro}, these methods do not scale well when the size of the problem increases. Motivated by these computational complexity considerations, a number of studies rely on approximation schemes and heuristics for solving \eqref{eq:intro}. \citet{rusmevichientong2009ptas}, \citet{Mittal2013general} and \citet{desir2014near} all propose approximation algorithms for assortment optimization under the MMNL model when the number of customer segments, $m$, is fixed. \citet{amiri1999bandwidth} develop a heuristic algorithm based on Lagrangian relaxation in the context of stochastic service systems. \citet{prokopyev2005multiple} present a GRASP-based (Greedy Randomized Adaptive Search) heuristic for solving the cardinality constrained problems. Finally, simple greedy algorithms are also used in the literature~\citep{feldman2015bounding, kunnumkal2015upper}. However, it is often not well understood when such algorithms perform well.

 {\textbf{Contributions and outline.}} The remainder of the note is organized as follows. In Section~\ref{sec:preliminaries}, we overview some necessary preliminaries and formulate our model \eqref{eq:intro} in terms of set functions.

In Section~\ref{sec:submodularity}, we provide the main result of the note that characterizes the \emph{submodularity} of a single ratio.  Submodularity is often a key property for devising approximation algorithms~\citep{fisher1978analysis,nemhauser1978analysis}. If the objective function can be identified as a submodular function, then simple greedy algorithms are capable of delivering high-quality solutions. In fact, it is possible to obtain $(1-e^{-1})$-approximations under a variety of feasible regions -- independently of the number of the ratios, $m$, involved--, thus improving over existing approximation methods for \eqref{eq:intro}.  We also discuss the connections  between submodularity and monotonicity in the context of fractional 0--1 optimization.

In Section~\ref{sec:application}, we consider our theoretical results in the context of two applications -- the assortment optimization  and the $p$-choice facility location problems. For the assortment optimization problem, our results suggest that submodularity is linked to a phenomenon known as \emph{cannibalization} \citep{moorthy1992market}, and naturally arises in several important scenarios. The results can also be applied in the case when there is a fixed cost associated with offering a product in the assortment \citep{Atamturk2017,kunnumkal2019tractable}, which arises, for example, in online advertisement with costs-per-impression. For the $p$-choice facility location problem~\cite{tawarmalani2002global}, we show how to reformulate the original problem in a desirable form that can be then exploited to benefit from the submodularity property. Finally, we conclude the note in Section~\ref{sec:conclusions}.

\section{Preliminaries}\label{sec:preliminaries}

{\textbf{Notation and additional assumption.}}Let $a^{k}=(a_{ki})_{i\in N}$ and $b^k=(b_{ki})_{i\in N\cup\{0\}}$ for all $k\in M$, and for given $a^k\in\R^{n}$ and $b^k\in \R^{n+1}$, define
\[ h(x;a^k,b^k):=\frac{\sum_{i\in N} a_{ki}x_i}{b_{k0}+\sum_{i\in N}b_{ki}x_i}.\]
 Then equation \eqref{eq:intro} can be rewritten as
\begin{equation}\label{eq:hForm}
	\max_{x\in\F}\;\sum_{k\in M}h(x;a^k,b^k).
\end{equation}
This form appears in many applications such as the retail assortment   and the $p$-choice facility location problems. Note that for each $x\in\{0,1\}^n$, there is a unique set $S=\{i\in N :\ x_i=1\}\subseteq N$, and conversely, each $S\subseteq N$ corresponds to an indicator vector $\indicator{S}\in\{0,1\}^n$.  Thus, we can rewrite $h(x;a^k,b^k)$ as a set function
\[h(S;a^k,b^k):=h(\indicator{S};a^k,b^k), \]
and regard $\F$ as the domain of sets, i.e., $\F\subseteq2^N$.
 Thereafter, we may use the vector form and the set form of \eqref{eq:hForm} interchangeably for convenience.

We also need the following additional assumption:

\vspace{1mm}

\textbf{A3}: \F is \emph{downward closed}, i.e., if $S\in \F$ then $T\in \F$ for all $T\subseteq S$.

\vspace{1mm}

\looseness-1We note that many types of feasible regions considered in the literature, such as $\F=2^N$~(unconstrained problem), $\F=\left\{S\subseteq N: \ |S|\leq p\right\}$ for some positive integer $p$ (cardinality constraint) and $\F=\left\{S\subseteq N:\  \sum_{i\in S}w_i\leq c\right\}$ for some weights $w\geq 0$ and $c\geq 0$ (capacity constraint) all satisfy Assumption \textbf{A3}.

\ignore{\subsection{MNL choice model}
Under the MNL choice model, each product $i\in N$ is associated with a weight $v_i> 0$, and the no-purchase option is associated with weight $v_0>0$; these weights encode the relative preferences of a customer for the products. Then, if assortment $S\subseteq N$ is offered, then the probability that a customer chooses product $i\in S$ is given by
$$q(i,S; v)=\frac{v_i}{v_0+\sum_{i\in S}v_i}.$$
The expected revenue from offering assortment $S\subseteq N$ is
$$r(S;v)=\sum_{i\in S}r_iq(i,S;v),$$
and the assortment optimization problem under the MNL model, which selects an assortment that maximizes the expected revenue, is given by
\begin{equation}\label{eq:MNL}\max_{S\in\F}r(S;v).\end{equation}
As mentioned in Section~\ref{sec:intro}, problem \eqref{eq:MNL} can be solved to optimality for a broad class of feasible regions \F, but suffers from limited modeling power.
}

 {\textbf{Submodularity and greedy algorithms.}} A set function $f:2^N\to \R$ from the subsets of $N$ to the real numbers is \emph{submodular} over $\mathcal{F}$ if it exhibits diminishing returns, i.e., $f(S\cup\{i\})-f(S)\geq f(T\cup\{i\})-f(T)$ for all $S\subseteq T\subseteq N\setminus\{i\}$ such that $T \cup\{i\}\in \F$. Equivalently, function $f$ is submodular over $\F$ if
\begin{equation}
\label{eq:submodular} f(S\cup\{i,j\})-f(S\cup\{j\})\leq f(S\cup\{i\})-f(S)
\end{equation}
for all $S\subseteq N$ and $i,j\not\in S$ such that $S\cup\{i,j\}\in \F$.

\looseness-1The greedy algorithm, see its pseudo-code in Algorithm~\ref{alg:greedy}, is a popular choice for tackling monotone submodular maximization problems because it is easy to implement and gives a constant-factor approximation in many cases. When the feasible region is a matroid, the greedy algorithm produces a solution with 1/2 approximation factor; see \cite{fisher1978analysis}. When the feasible region is given by a cardinality constraint,  the approximation ratio can be improved to ($1-e^{-1}$); see \cite{nemhauser1978analysis}. Other ($1-e^{-1}$)-approximation algorithms or near-optimal algorithms have also been provided for other classes of feasible regions over the years \cite{calinescu2011maximizing,kulik2009maximizing,sviridenko2004note}, for example, when $\F$ is defined with a single or multiple capacity constraints.
\begin{algorithm}
	\caption{Greedy Algorithm for Submodular Function Maximization}
	\label{alg:greedy}
\begin{enumerate}
	\item[\emph{Step 1.}] Set $S:=\emptyset$.
	\item[\emph{Step 2.}] Set $A:=\{\ell\in N\setminus S:\ S\cup \{\ell\}\in \F\}$.
	\item[\emph{Step 3.}] If $A\neq \emptyset$, set $\ell^*\in\argmax_{\ell\in A}f(S\cup\{\ell\})$ and $S:=S\cup\{\ell^*\}$. Go to \emph{Step 2}.
	\item[\emph{Step 4.}] Return $S$.
\end{enumerate}
\end{algorithm}


\section{Submodularity of a single ratio and its implications} \label{sec:submodularity}

\subsection{A necessary and sufficient condition}
In this section, we give a necessary and sufficient condition for the submodularity of the function $h(\cdot)$, see Theorem~\ref{prop:iff_cert}. As a direct consequence, if $h(\cdot;a^k,b^k)$ satisfies the condition of Theorem~\ref{prop:iff_cert} for every $k\in M$, then it follows that the fractional 0--1 program \eqref{eq:hForm} admits a constant-factor approximation algorithm. For convenience, we drop the superscript $k$ in $a^k$ and $b^k$ and use the notation $h(\cdot;a,b)$ throughout this section.

We first consider the case where $b_0>0$. The key result of this note is as follows:
\begin{theorem}
	\label{prop:iff_cert}
If $b_0>0$, then function $h(\cdot;a,b)$	is submodular over $\F$ if and only if
	\begin{equation}
		\label{eq:iff cert}
		h(S\cup\{i\};a,b)+h(S\cup\{j\};a,b)\leq \frac{a_i}{b_i}+\frac{a_j}{b_j}
	\end{equation}
	for all $S\subseteq N$, and $i,j\not\in S$ with $i\neq j$ such that $S\cup\{i\}\cup\{j\}\in \F$.
\end{theorem}
\begin{proof}
	Recall that Assumption \textbf{A2} holds. Thus, the right-hand side of (\ref{eq:iff cert}) is well-defined. Let $S\subseteq N$, let $i,j\not\in S$ with $i\neq j$ satisfying $S\cup\{i\}\cup\{j\}\in \F$, and define $A_S = \sum_{j\in S} a_j$ and $B_S = b_0 + \sum_{j\in S} b_j$. Observe that $h(S;a,b)={A_S}/{B_S}$. From \eqref{eq:submodular} we find that $h(\cdot;a,b)$ is submodular if and only if
	\begin{align*} &\frac{A_S+a_i+a_j}{B_S+b_i+b_j}-\frac{A_S+a_j}{B_S+b_j}\leq \frac{A_S+a_i}{B_S+b_i} -\frac{A_S}{B_S}.
\end{align*}
	Multiplying both sides by $B_S(B_S+b_i+b_j)$, we get the equivalent condition
	\begin{align*}
	&B_S\left(A_S+a_i+a_j\right)-B_S\left(1+\frac{b_i}{B_S+b_j}\right)(A_S+a_j)\\
	&\leq B_S\left(1+\frac{b_j}{B_S+b_i}\right)(A_S+a_i)-(B_S+b_i+b_j)A_S\\
		\Leftrightarrow\;&a_iB_S-b_iB_S\frac{(A_S+a_j)}{B_S+b_j}
		\leq a_iB_S-b_iA_S-b_jA_S+\frac{b_j}{B_S+b_i}B_S(A_S+a_i)\\
		\Leftrightarrow\;&a_iB_S-b_iB_S\frac{(A_S+a_j)}{B_S+b_j}
		\leq a_iB_S-b_iA_S+\frac{b_j}{B_S+b_i}\Big(B_SA_S+B_Sa_i-(B_S+b_i)A_S\Big)\\
		\Leftrightarrow\;&a_iB_S-b_iB_S\frac{(A_S+a_j)}{B_S+b_j}
		\leq a_iB_S-b_iA_S+\frac{b_j}{B_S+b_i}(a_iB_S-b_iA_S).
	\end{align*}
	Adding $b_iA_S +b_iB_S\frac{A_S+a_j}{B_S+b_j}-a_iB_S$ to both sides, we find
	\begin{align*}
	&b_iA_S \leq b_iB_S\frac{(A_S+a_j)}{B_S+b_j}+\frac{b_j}{B_S+b_i}(a_iB_S-b_iA_S)\\
	\Leftrightarrow\; &b_iA_S\left(B_S+b_i\right)\left(B_S+b_j\right) \leq b_iB_S(A_S+a_j)(B_S+b_i)+b_j(B_S+b_j)(a_iB_S-b_iA_S)\\
	\Leftrightarrow\;&b_iA_SB_S^2+b_iA_SB_S(b_i+b_j)+b_i^2b_jA_S \\
	&\leq b_iA_SB_S^2+b_ia_jB_S^2+B_SA_Sb_i^2+a_jb_i^2B_S+a_ib_jB_S^2+a_ib_j^2B_S-b_ib_jA_SB_S-b_ib_j^2A_S.
	\end{align*}
	After rearranging and canceling out some terms in the above expression, we obtain:
	\begin{align*}
&2b_ib_jA_SB_S+b_i^2b_jA_S+b_ib_j^2A_S \leq B_S^2(a_ib_j+a_jb_i)+a_jb_i^2B_S+a_ib_j^2B_S\\
		\Leftrightarrow\;&b_ib_jA_S(b_i+b_j+2B_S)\leq b_ib_j\left(a_j/b_j B_S^2+a_j/b_jb_iB_S+a_i/b_iB_S^2+a_i/b_ib_jB_S\right)\\
		\Leftrightarrow\;&A_S(b_i+b_j+2B_S)\leq a_j/b_j B_S^2+a_j/b_jb_iB_S+a_i/b_iB_S^2+a_i/b_ib_jB_S\\
		\Leftrightarrow\;&(B_S+b_i)(A_S-a_j/b_jB_S)+(B_S+b_j)(A_S-a_i/b_iB_S)\le 0.
	\end{align*}
	
	Finally, dividing by $(B_S+b_i)(B_S+b_j)$ and then adding $a_i/b_i+a_j/b_j$ on both sides, we get
	\begin{align*}\Leftrightarrow\;&\left(\frac{A_S-a_j/b_jB_S}{B_S+b_j}+a_j/b_j\right) + \left(\frac{A_S-a_i/b_iB_S}{B_S+b_i}+a_i/b_i\right)\leq a_i/b_i+a_j/b_j\\
	\Leftrightarrow\;&\frac{A_S+a_j}{B_S+b_j}+\frac{A_S+a_i}{B_S+b_i} \leq a_i/b_i+a_j/b_j,
	\end{align*}
	which is precisely inequality \eqref{eq:iff cert}.
\end{proof}
As we discuss next, submodularity is closely linked to monotonicity.
\subsection{Monotonicity implies submodularity}\label{subsec:monotonic}
\looseness-1The function $h(\cdot;a,b)$ is \emph{monotone nondecreasing} if \begin{equation}
\label{eq:monotonicity}
h(S;a,b)\leq h(S\cup\{j\};a,b)\end{equation} for every set $S$ and $j\not\in S$ such that $S\cup\{j\}\in \F$. Monotonicity is often a prerequisite for greedy algorithms, see, e.g., \cite{nemhauser1978analysis}, to guarantee a constant approximation factor.  Also, it arises naturally in many applications; see Section~\ref{sec:cannibalization} for details. As we show next, monotonicity is a sufficient condition for submodularity.
\begin{proposition}
	\label{prop:monotonicity}
	If function $h(\cdot;a,b)$ is monotone nondecreasing, then $h(\cdot;a,b)$ is submodular.
\end{proposition}
\begin{proof}
	Condition \eqref{eq:monotonicity} is equivalent to
	\begin{align}
	&\frac{\sum_{i\in S}a_i }{b_0+\sum_{i\in S}b_i}\leq \frac{\sum_{i\in S}a_i +a_j}{b_0+\sum_{i\in S}b_i+b_j}\notag\\
	\Leftrightarrow \;&\left(1+\frac{b_j}{b_0+\sum_{i\in S}b_i}\right)\sum_{i\in S}a_i\leq \sum_{i\in S}a_i +a_j\notag\\
	\Leftrightarrow \;&\frac{\sum_{i\in S}a_i}{b_0+\sum_{i\in S}b_i}\leq \frac{a_j}{b_j}\notag\\
	\Leftrightarrow \;&h(S;a,b)\leq \frac{a_j}{b_j}\label{eq:monotonicity2}
	\end{align}
	for all $S$ and $j\not\in S$. Therefore, if $i,j\not\in S$, then $h(S\cup \{i\};a,b)\leq {a_i}/{b_i}$ and $h(S\cup \{j\};a,b)\leq {a_j}/{b_j}$, and inequality \eqref{eq:iff cert} follows.
\end{proof}
Inequality~\eqref{eq:monotonicity2} needs to hold for every combination of set $S$ and element $i$ for the function to be monotone. Thus, checking monotonicity can be done by verifying that
\begin{equation}\label{eq:sufficient}
\min_{i\in N}\frac{a_i}{b_i} \ge \max_{S\in\F}h(S;a,b),
\end{equation}
and the optimization problem on the right-hand side of \eqref{eq:sufficient} can be solved using existing algorithms for single-ratio fractional optimization; see \cite{megiddo1979combinatorial,Radzik1998}. In fact, in some cases monotonicity can be verified without solving an optimization problem.
\begin{corollary}\label{cor:monotonicity}
	\looseness-1Function $h(\cdot;a,b)$ is monotone nondecreasing (and submodular) over $2^N$ if and only if $$\min_{i\in N}\frac{a_i}{b_i}\geq h(N;a,b).$$
\end{corollary}
\begin{proof}
	The forward direction follows directly from \eqref{eq:sufficient}. For the backward direction, let ${a^*}/{b^*}=\min_{i\in N}\{a_i/b_i\}$, and then we find that
	\begin{align*}
	h(N;a,b)\leq \frac{a^*}{b^*}
	&\Leftrightarrow \frac{\sum_{i\in N}a_i}{b_0+\sum_{i\in N}b_i}\leq \frac{ a^*}{b^*}\Leftrightarrow \sum_{i\in N}\left(\frac{a_i}{b_i}-\frac{a^*}{b^*}\right)b_i\le \frac{ a^*}{b^*}b_0.
	\end{align*}
	Since ${a_i}/{b_i}\ge {a^*}/{b^*}$, we find that $\sum_{i\in S}({a_i}/{b_i}-{a^*}/{b^*})b_i\le \left({a^*}/{b^*}\right)b_0$ for any $S\subseteq N$, i.e., $h(S;a,b)\le {a^*}/{b^*}$ for any $S\subseteq N$.
\end{proof}

\subsection{On non-monotone submodular functions}
From Proposition~\ref{prop:monotonicity}, we know that monotonicity implies submodularity. In general, as Example~\ref{ex:counter} below shows, the converse does not hold.
\begin{example}\label{ex:counter}Assume we have three variables, i.e., $N=\{1,2,3\}$, with the setting $(a_1,a_2,a_3)=(3,2,1)$ and $(b_0,b_1,b_2,b_3)=(2,1,1,1)$. Then from Theorem~\ref{prop:iff_cert} we can verify that $h(\cdot;a,b)$ is submodular over $2^N$: since $a_i/b_i+a_j/b_j\geq 3$ for any $i\neq j$ and, for any $S\subseteq N$, $h(S;a,b)\leq h(\{1,2\};a,b)=5/4\leq 3/2$, we find that inequality \eqref{eq:iff cert} holds. However, $h(\{3\};a,b)=1/3<h(\{1,2,3\};a,b)=6/5<h(\{1,2\};a,b)=5/4$, and monotonicity does not hold.\qed
\end{example}	
Nonetheless, if $h(\cdot;a,b)$ is submodular, then it is in fact very close to a nondecreasing function as shown in Proposition~\ref{prop:nonMonotone} below. In particular, if the decision variable with the smallest value $a_i/b_i$ is fixed, then the resulting function is monotone.

Assume for the remainder of this section, without loss of generality, that $a_1/b_1\ge a_2/b_2 \ge \dots \ge a_n/b_n$. Define $\F_1:=\{ S\in \F:\ n\in S \}$ and $\F_2:=\{ S\in \F:\ n\notin S \}$.

\begin{proposition}\label{prop:nonMonotone}
	If $h(\cdot;a,b)$ is submodular over $\F$, then the following holds:\\
	(\textit{i}) \ function $h(\cdot;a,b)$ is monotone nondecreasing over $\F_1$;\\
	(\textit{ii}) for any $S\in\F_2$ and any $j\neq n$ such that $S\cup\{j\}\in \F$ and $S\cup\{n\}\in \F$, we have $h(S\cup\{j\};a,b)\ge h(S;a,b)$.
\end{proposition}
\begin{proof}
	\looseness-1We first prove $h(\cdot;a,b)$ is monotone nondecreasing over $\F_1$ by contradiction. Assume there exists $S$ and $j\neq n$ such that $n\notin S$ and $h(S\cup\{j,n\};a,b)<h(S\cup\{n\};a,b)$. Because $h(S\cup\{j,n\};a,b)$ is a convex combination of $h(S\cup\{n\};a,b)$ and ${a_j}/{b_j}$, we have ${a_j}/{b_j}<h(S\cup\{n\};a,b)$. Since ${a_n}/{b_n}\le {a_j}/{b_j}$, we find that ${a_n}/{b_n}<h(S\cup\{n\};a,b)$. Note that
	\begin{align*}
	h(S\cup\{n\};a,b)&=\left(\frac{b_0+\sum_{i\in S}b_i}{b_0+b_n+\sum_{i\in S}b_i}\right)h(S;a,b)+\frac{b_n}{b_0+b_n+\sum_{i\in S}b_i}\frac{a_n}{b_n}
	\end{align*}
	 is a convex combination of $h(S;a,b)$ and $a_n/b_n$, and since $a_n/b_n<h(S\cup\{n\};a,b)$, it follows that $h(S\cup\{n\};a,b)<h(S;a,b)$. By submodularity, $h(S\cup\{j,n\};a,b)-h(S\cup\{j\};a,b)\le h(S\cup\{n\};a,b)-h(S;a,b)<0$, which indicates that $h(S\cup\{j,n\};a,b)<h(S\cup\{j\};a,b)$. Thus, ${a_n}/{b_n}<h(S\cup\{j\};a,b)$. However, this implies $h(S\cup\{j\};a,b)+h(S\cup\{n\};a,b)>{a_j}/{b_j}+{a_n}/{b_n}$, which is a contradiction based on Theorem~\ref{prop:iff_cert}. Thus, (\textit{i}) holds.
	
	Next, we prove (\textit{ii}) by contradiction. Assume there exists $S$ and $j\neq n$ such that $n\notin S$ and $h(S\cup\{j\};a,b)<h(S;a,b)$. Because $h(S\cup\{j\};a,b)$ is the weighted average of $a_j/b_j$ and $h(S;a,b)$,  we have that ${a_j}/{b_j}<h(S\cup\{j\};a,b)<h(S;a,b)$. Recall that ${a_n}/{b_n}\le {a_j}/{b_j}$. Hence,  ${a_n}/{b_n}<h(S;a,b)$, which implies ${a_n}/{b_n}<h(S\cup\{n\};a,b)$ -- using similar arguments as in the proof of (\textit{i}). Hence, $h(S\cup\{j\};a,b)+h(S\cup\{n\};a,b)>{a_j}/{b_j}+{a_n}/{b_n}$, which contradicts the submodularity of $h(\cdot;a,b)$.	
\end{proof}
\begin{corollary}
	\label{cor:nearmonotonicity1} If either $\F=2^N$ or $\F=\{S\subseteq N:\ |S|\leq p\}$ for any $p\in\{1,\ldots,n-1\}$, then submodularity of $h(\cdot;a,b)$ over $\F$ implies that $h(\cdot;a,b)$ is monotone nondecreasing over $\F_1$ and $\F_2$.
\end{corollary}

\setcounter{example}{0}


\begin{example}[Continued] Observe that $h(\cdot;a,b)$ is indeed monotone over $\F_1$, since $h(\{3\};a,b)=1/3$, $h(\{1,3\};a,b)=1$, $h(\{2,3\};a,b)=3/4$ and $h(\{1,2,3\};a,b)=6/5$. Similarly, we can verify that $h(\cdot;a,b)$ is monotone over $\F_2$ since $h(\emptyset;a,b)=0$, $h(\{1\};a,b)=1$, $h(\{2\};a,b)=2/3$ and $h(\{1,2\};a,b)=5/4$.\qed
\end{example}

\subsection{On homogeneous fractional functions}
In this section, we show that the assumption $b_0>0$ is indeed necessary in Theorem~\ref{prop:iff_cert}, as otherwise submodularity does not hold in most practical situations.
Proposition~\ref{prop:homo} below formalizes this statement.

\begin{proposition}\label{prop:homo}
	Assume $b_0=0$. If there exists a feasible set $S$ such that there are at least three distinct values for ${a_i}/{b_i},\ i\in S$, then $h(\cdot;a,b)$ is not submodular.
\end{proposition}
\begin{proof}
	Assume without loss of generality that ${a_1}/{b_1}<{a_2}/{b_2}<{a_3}/{b_3}$. Then the following inequality 	\[\frac{b_{1}}{b_{1}+b_{3}}\left(\frac{a_3}{b_3}-\frac{a_1}{b_1}\right)+\frac{b_{2}}{b_{2}+b_{3}}\left(\frac{a_3}{b_3}-\frac{a_2}{b_2}\right)\ge \frac{b_{1}}{b_{1}+b_{2}+b_{3}}\left(\frac{a_3}{b_3}-\frac{a_1}{b_1}\right) +\frac{b_{2}}{b_{1}+b_{2}+b_{3}}\left(\frac{a_3}{b_3}-\frac{a_2}{b_2}\right).\]
	holds since denominators are greater on the right-hand side.
	Subtracting $2\cdot\left({a_3}/{b_3}\right)$ on both sides, we find that
	\[-\frac{a_1+a_3}{b_1+b_3}-\frac{a_2+a_3}{b_2+b_3}\ge -\frac{a_1+a_2+a_3}{b_1+b_2+b_3}-\frac{a_3}{b_3}, \]
	which is equivalent to $h(\{1,3\};a,b)+ h(\{2,3\};a,b)\le h(\{1,2,3\};a,b)+ h(\{3\};a,b)$, violating the definition of  submodularity.
\end{proof}

\section{Applications}\label{sec:application}
In this section, we discuss the implications of our theoretical results in the context of the assortment optimization and the $p$-choice facility location problems.
\subsection{Assortment optimization problem}
In the assortment optimization problem, a firm offers a set of products to utility-maximizing customers. The goal of the firm is to choose an assortment of products that maximizes its expected revenue. It is a core revenue management problem pervasive in practice~\citep{talluri2004revenue}. In this subsection, we mainly consider this problem under the mixed multinomial logit model~(MMNL); see, e.g.,  \citep{bonnet2001assessing,mcfadden2000mixed}.

Formally, let $N$ be the set of products that can be offered to customers. Denote by $r_i$ the revenue perceived by the firm if a customer chooses product $i\in N$. Under the MMNL model, each product $i\in N$ is associated with a random weight $v_{ki}>0$, and the no-purchase option is associated with weight $v_{k0}>0$; these weights encode the relative preferences for the products by a customer of type $k\in M$, i.e., set $M$ describes market segments.

Given the preference weights $v^k$, if assortment $S\subseteq N$ is offered, then the probability that a customer in $k\in M$ chooses product $i\in S$ is given by
\[q(i,S;v^k)=\frac{v_{ki}}{v_{k0}+\sum_{i\in S}v_{ki}}. \]
The conditional expected revenue from offering assortment $S\subseteq N$ is
\[ r(S;v^k)=\sum_{i\in S}r_iq(i,S;v^k). \]
Taking the expectation over the random vector $v^k$, we formulate the assortment optimization problem under the MMNL model as
\begin{equation}
\max_{S\in\F}\E_{ v} \left[r(S;v)\right]=\sum_{k\in M}p_k r(S;v^k),\label{eq:mixedMNL}
\end{equation}
where $p_k$ is the probability of a customer to be in segment $k$ and
each realization of $v$ can be interpreted as the preferences associated with a given customer of customer segment. We assume that the support of $v$ is finite. Hence, \eqref{eq:mixedMNL} can be posed in the form of \eqref{eq:intro}, where $a_{ki}= p_k r_i v_{ki}$, $b_{ki}=v_{ki}$ and $b_{k0}=v_{k0}$ for all $k\in M$ and $i\in N$. Thus, ${a_{ki}}/{b_{ki}}=p_kr_i$.

Finally, we note that $p_k\geq 0$ for each $k\in M$. Hence, for submodularity of the objective function in (\ref{eq:mixedMNL}) it is sufficient to consider the single-ratio functions $r(\cdot;v^k)$, $k\in M$. Therefore, in our discussion below when applying the results of Theorem \ref{prop:iff_cert} and Corollary \ref{cor:monotonicity} (with ratio $a_i/b_i$), the multiplier $p_k$ can be dropped from consideration.

\subsubsection{Cannibalization and submodularity}\label{sec:cannibalization}

Intuitively, in retail assortment problems, monotonicity of the revenue function implies that there is limited cannibalization, i.e., the introduction of a new product $i$ (when feasible) always increases the expected revenue perceived by the firm -- despite that the revenue obtained from previously offered products in $S$ might decrease slightly. To be more specific, this limited cannibalization phenomenon arises in online advertising: the probability that a given customer clicks on an ad is often quite low, and the advertiser usually profits from offering more ads within the limited number of spots on the webpage.

Let $r_{\min}=\min_{i\in N} r_i$ and $r_{\max}=\max_{i\in N} r_i$. By Proposition~\ref{prop:monotonicity} and Corollary~\ref{cor:monotonicity}, we obtain the following results in terms of revenue functions immediately.
 \begin{corollary}
 	If function $r(\cdot;v)$ is monotone nondecreasing, then $r(\cdot;v)$ is submodular.
 \end{corollary}
\begin{corollary}
	Function $r(\cdot;v)$ is monotone nondecreasing (and submodular) over $2^N$ if and only if $r_{min}\geq r(N;v)$.
\end{corollary}
 \subsubsection{Revenue spread, no-purchase probability and submodularity}
When the revenues $r$ of all products are identical, assortment optimization problems are known to be submodular maximization problems \citep{Atamturk2017,desir2015capacity}. Intuitively, one would expect that if the revenues are sufficiently close (but not identical), then submodularity should be preserved. Proposition~\ref{prop:assortment_dispersion} formalizes this intuition: if the gap between the largest and the smallest revenues is bounded above by the odds of no-purchase, then the function is nondecreasing and submodular.
\begin{proposition}
	\label{prop:assortment_dispersion}
	If
	\begin{equation}\label{eq:revenue} \frac{r_{\max}-r_{\min}}{r_{\min}} \leq \min\limits_{S\in \F}\frac{1-q(S;v)}{q(S;v)}, \end{equation}
	then $r(\cdot;v)$ is nondecreasing and submodular, where $r_{\max}$ and $r_{\min}$ are the largest and smallest revenues, respectively, and $q(S;v)=\sum_{i\in S}q(i,S;v)$ is the probability that an item is purchased.
\end{proposition}
\begin{proof}
	Equation \eqref{eq:revenue} can be rewritten as $r_{\max}q(S;v)\leq r_{\min}$ for all $S\in \mathcal{F}$. Since for any $S$ and $i\not\in S$ it follows that $r(S;v)\leq r_{\max}q(S;v)\leq r_{\min}\leq r_i$, we find that \eqref{eq:monotonicity2} is satisfied and the function $r(\cdot;v)$ is monotone submodular.
\end{proof}

\looseness-1Proposition~\ref{prop:assortment_dispersion} provides us with additional intuition on the industries in which the expected revenues are submodular functions of the assortment offered. In the online advertisement, where the revenues obtained from clicks are usually similar and the odds of no-purchase are high, we would expected to obtain submodular revenue functions. In a \emph{monopoly}, the firm offering the assortment would have a large flexibility in setting prices (resulting in a large revenue spread) and the odds of no-purchase would be low (due to the lack of competing alternatives), resulting in a revenue function that is not submodular. In contrast, in a \emph{competitive market}, the odds of no-purchase would be larger and firms have little or no control over prices (and if the values $r_i$ are interpreted as profits instead of revenues, the spread would typically be low), resulting in submodular revenue functions.

\looseness-1From Proposition~\ref{prop:assortment_dispersion} we also gain insights on the differences between revenue management in the airline and hospitality industries, two industries that are often treated as equivalent in the literature \citep{talluri2006theory}. In the hospitality industries, no-purchase odds can be high as shown by the relatively low occupancy rates -- 66.1\% in the US \citep{OccupancyUS}
in 2018; in addition, revenue differences between products are often due to ancillary charges (e.g., breakfast, non-refundable, long stay), which account for a small portion of the baseline price for a room. In such circumstances we would expect revenue functions to be submodular and simple greedy heuristics to perform well. In contrast, in the airline industries no-purchase odds are often smaller -- the load factor was 86.1\% in the US in 2018 \citep{LoadFactor} --, and air fares can change dramatically depending on the conditions. Thus, in the airline industry we would expect to encounter non-submodular revenue functions, and simple heuristics may be inadequate.

\subsubsection{On the greedy algorithm and revenue-ordered assortments}
Revenue-ordered assortments are optimal for unconstrained assortment optimization under the MNL model, and tend to perform well in practice \citep{talluri2004revenue}. \citet{berbeglia2017assortment} study the revenue-ordered assortments under the general discrete choice model and prove performance guarantees.
\begin{proposition}[\citet{berbeglia2017assortment}] \label{prop:revenueOrdered}Revenue-ordered assortments are a  $\frac{1}{1+\log\left(\frac{r_{\max}}{r_{\min}}\right)}$-approxi\-mation for the  {unconstrained} assortment optimization problem under the MMNL choice model, where $r_{\max}$ and $r_{\min}$ are the largest and smallest revenues, respectively.
\end{proposition}
Thus, the quality of revenue-ordered assortments depend on the ratio $r_{max}/r_{min}$; in particular, if $r_{max}/r_{min}=1$, then the revenue-ordered assortments strategy delivers an optimal solution, and the guarantee degrades as the value of the ratio increases.

From Proposition~\ref{prop:assortment_dispersion}, we can also obtain guarantees depending on the ratio $r_{max}/r_{min}$. Define:
\begin{equation}\label{eq:maxAlpha}\alpha(S)=\max_{k\in M}q(S; v^k)=\max_{k\in M}\sum_{i\in S}q(i,S;v^k)\end{equation}
as the maximum probability that a customer from any segment purchases an item when assortment $S$ is offered.
\begin{proposition}\label{prop:ratio}\label{p} If $\F=\{ S:\ |S|\le p \}$ for some positive integer $p$ and $r_{max}/r_{min}\le 1+\frac{1-\alpha(S)}{\alpha(S)}$ for all $S\in\F$, then Algorithm~\ref{alg:greedy} delivers a $(1-1/e)$-optimal solution for the assortment optimization problem under the MMNL choice model.
\end{proposition}

\looseness-1Unlike Proposition~\ref{prop:revenueOrdered}, we impose a condition on the ratio $r_{max}/r_{min}$ in Proposition~\ref{prop:ratio}; however, if such condition is satisfied, then we obtain an approximation guarantee of $(1-1/e)\approx 0.63$ for the more general assortment optimization problem under a cardinality constraint. 

\looseness-1Finally, we also point out that \citet{rusmevichientong2014assortment} prove that if customers are \emph{value conscious}, i.e., ${v}_1\leq  v_2\leq\ldots\leq  v_n$ and $r_1 v_1\geq r_2 v_2\geq \ldots\geq r_n v_n$ for all realizations of ${v}$, then the revenue ordered assortments are optimal for the unconstrained and cardinality constrained cases. It is easy to check that in this case the solutions obtained from the greedy algorithm correspond precisely with the revenue ordered assortments. Thus, Algorithm~\ref{alg:greedy} delivers optimal solutions as well.

\ignore{
\subsubsection{On submodularity and the Markov chain choice model}
\citet{blanchet2016markov} propose the Markov chain choice model as a generalization of several existing choice models in the literature that admits tractable solution techniques. Moreover, they show that under some conditions, the Markov chain choice model can approximate well the MMNL choice model (and thus any choice model). As we now discuss, the approximation quality of the Markov chain choice model is closely linked to the submodularity of the revenue function.
\begin{proposition}[\citet{blanchet2016markov}]
	\label{prop:MCchoice} For any $S\subseteq N$ and $j\in S$, let $q_{mc}(i,S)$ as the probability that a customer selects item $i\in N$ given that assortment $S$ is offered under the Markov chain choice model. Then
	$$\big(1-\alpha(N\setminus S)^2\big)q_{mc}(i,S)\leq \mathbb{E}_{ v}\left[q(i,S; v)\right]\leq \left(1+\tau\frac{\alpha(N\setminus S)}{1-\alpha(N\setminus S)}\right)q_{mc}(i,S), $$
	where $\alpha(S)$ is given as in \eqref{eq:maxAlpha} and $\tau$ is a parameter that depends on the Markov chain choice model.
\end{proposition}
In other words, the Markov chain choice model is able to accurately approximate the MMNL model when the probabilities of a purchase $\alpha(S)$ are sufficiently small. However, from Propositions~\ref{prop:assortment_dispersion} and \ref{prop:ratio} we see that if this is the case, then all functions $r(S;v)$ are submodular. Finally, note that the best known approximation algorithms for constrained assortment optimization under the Markov chain choice model have approximation guarantees of at most $1/2$ \citep{desir2015capacity}.

Therefore, we conclude that in constrained assortment optimization problems under the MMNL choice model in which the Markov chain choice model is able to accurately approximate the MMNL, tackling directly the MMNL choice model via the greedy  algorithm may yields better bounds than using the Markov chain proxy.
}

\subsection{$p$-choice facility location problem}
Facility location problems deal with deciding where to locate facilities across a finite set of feasible points, taking into account the needs of customers to be served in such a way that a given economic index is optimized~\cite{barros2013discrete}. Submodularity often arises in facility location problems; see \cite{atamturk2012conic, du2012primal, li2015improved,ortiz2017formulations}. In this subsection, we consider a particular class of facility location problems with a fractional 0--1 objective function, referred to as the $p$-choice facility location problem, which is considered in \cite{tawarmalani2002global}. In the $p$-choice facility location problem, a decision-maker has to decide where to locate $p$ facilities in $n$ possible locations to service $m$ demand points, in order to maximize the market share.


Formally, let $d_k>0$ be the demand at customer location $k\in M=\{1,\dots,m\}$, and $v_{ki}>0$ be the utility of location $i$ to customers at $k$. Let $S\subseteq N:=\{1,\ldots,n\}$, $|S|=p$, be the set of facilities chosen by the decision-maker. It is assumed that the market share provided by facility $j\in S$ with respect to demand point $k$ is given by:
\[ d_k\frac{v_{kj}}{\sum_{i\in S}v_{ki}}. \]
Let $w_i>0$ be some weight parameter that represents the importance of locating facility in location $i\in N$. Then
the problem of determining the set of facility locations $S$ that maximizes the weighted market share can be formulated as:
\[ \max_{|S|=p}\;\sum_{i\in S}w_i\sum_{k\in M}d_k\frac{v_{ki}}{\sum_{i\in S}v_{ki}}, \]
which can be reorganized as
\begin{equation}\label{eq:pfl}
	\max_{|S|=p}\;\sum_{k\in M}d_k\frac{\sum_{i\in S}v_{ki}w_i}{\sum_{i\in S}v_{ki}}.
\end{equation}
Clearly, the model in \eqref{eq:pfl} can be formulated as a fractional 0--1 program given by (\ref{eq:intro}).


Note that from Proposition~\ref{prop:homo}, the objective function in \eqref{eq:pfl} is, in general, not  submodular since it is homogeneous.
Nonetheless, exploiting the equality constraint, we can convert the objective function to a non-homogeneous one. Define $v_{\min}^k=\delta\cdot \min_{i\in N}\{v_{ki}\}$ for some fixed $\delta\in(0,1)$. For any feasible solution $S$, where $|S|=p$, we also have that:
\[ \sum_{i\in S}v_{ki}=\sum_{i\in S}v_{\min}^k+\sum_{i\in S}(v_{ki}-v_{\min}^k)=pv_{\min}^k+\sum_{i\in S}(v_{ki}-v_{\min}^k). \]
 As a result, \eqref{eq:pfl} can be equivalently stated as:
\begin{equation}\label{eq:relaxedpfl}
\max_{|S|= p}\;\sum_{k\in M}d_k\frac{\sum_{i\in S}v_{ki}w_i}{pv_{\min}^k+\sum_{i\in S}(v_{ki}-v_{\min}^k)},
\end{equation}
\noindent where $v_{ki}-v_{\min}^k>0$ and $v_{\min}^k>0$ for all $i\in N$ and $k\in M$ by our construction procedure.

\looseness-1Recall our discussion on the links between monotonicity and submodularity in Section~\ref{subsec:monotonic}.
Applying inequality (\ref{eq:sufficient}), we find that a given ratio $k$ in the objective function of (\ref{eq:relaxedpfl}) is monotone nondecreasing over set $\mathcal{F}:=\{S \subseteq N:\ |S|\leq p\}$ if
\begin{equation}\label{eq:pChoice}\min_{i\in N}\frac{v_{ki}w_i}{v_{ki}-v_{\min}^k}\geq \max_{|S|\leq p}\frac{\sum_{i\in S}v_{ki}w_i}{pv_{\min}^k+\sum_{i\in S}(v_{ki}-v_{\min}^k)}.\end{equation}
Hence, if \eqref{eq:pChoice} holds for all ratios $k\in M$, then the feasibility set in \eqref{eq:relaxedpfl} can be relaxed to $|S|\leq p$. Consequently, Assumption \textbf{A3} is satisfied and \eqref{eq:relaxedpfl} reduces to the maximization problem of a submodular function by Proposition~\ref{prop:monotonicity}.

The right-hand side of \eqref{eq:pChoice} can be interpreted as the best average revenue weighted by market share, or simply the best total revenue that can be obtained from customer segment $k$. The intuition for (\ref{eq:pChoice}) to hold in the $p$-choice facility location problem is rather similar to our observations in the assortment optimization problem. Indeed, it is easy to verify, for example, that if all locations have the same utilities and weights, i.e., $v_{ki}=v^k$ and $w_i=w$ for all $i\in N$ and some $v^k$ and $w$, then (\ref{eq:pChoice}) holds. Moreover, from \eqref{eq:pChoice} we obtain the following sufficient condition.

\begin{proposition}
	Let $w_{\max}$ and $w_{\min}$ be the maximum and minimum weights, and let $v_{\max}^k$ be the maximum utility associated with customer segment $k$. If \begin{equation}\label{eq:pchoiceSufficient}
	\frac{w_{\min}}{w_{\max}}+1\geq \frac{v_{\max}^k}{v_{\min}^k},
	\end{equation}
 then the revenue of customer segment $k$ is submodular.
\end{proposition}
\begin{proof}
	Observe that since $w_i\geq w_{\min}$ and $\frac{v_{ki}}{v_{ki}-v_{\min}^k}\geq \frac{v_{\max}^k}{v_{\max}^k-v_{\min}^k}$, we find that $$\frac{v_{ki}w_i}{v_{ki}-v_{\min}^k}\geq \frac{v_{\max}^k}{v_{\max}^k-v_{\min}^k}w_{\min}.$$ Moreover, we also find that
	$$\max_{|S|\leq p}\frac{\sum_{i\in S}v_{ki}w_i}{pv_{\min}^k+\sum_{i\in S}(v_{ki}-v_{\min}^k)}\leq \max_{|S|\leq p}\frac{w_{\max}\sum_{i\in S}v_{ki}}{pv_{\min}^k}\leq \frac{w_{\max}v_{\max}^k}{v_{\min}^k}. $$
	After rearranging terms corresponding to the sufficient condition
	$$\frac{v_{\max}^k}{v_{\max}^k-v_{\min}^k}w_{\min}\geq \frac{w_{\max}v_{\max}^k}{v_{\min}^k},$$ we obtain precisely \eqref{eq:pchoiceSufficient}.
\end{proof}

Simply speaking, if the considered facility locations are sufficiently similar with respect to their utilities, i.e., $\frac{v_{\max}^k}{v_{\min}^k}\approx 1$, then ratios in \eqref{eq:relaxedpfl} are submodular. Submodularity may be preserved for larger spread of utilities, provided that the weights are sufficiently close. If all the considered facility locations are sufficiently similar with respect to their utilities and weights, then \eqref{eq:pfl} can be reduced to maximizing a submodular function; consequently, high-quality solutions can be obtained by a greedy approach, e.g., Algorithm~\ref{alg:greedy}.

 \ignore{Simply speaking, if the considered facility locations are sufficiently similar with respect to their utilities and weights, then \eqref{eq:pfl} can be reduced to maximizing a submodular function; consequently, high-quality solutions can be obtained by a greedy approach, e.g., Algorithm~\ref{alg:greedy}.}


\ignore{Inequality \eqref{eq:sufficient} is easily satisfied for low-utility locations, since the lhs of \eqref{eq:pChoice} goes to infinity as $v_{i\ell}\to v_{min}^i$. Finally, inequality \eqref{eq:pChoice} is satisfied for high-utility locations if the corresponding revenue is large as well \hsn{I guess this may be not true if the inequality is reversed.}. Therefore, we see that the revenue function corresponding to a customer segment is submodular if locations that command a large market share have large revenues as well.}

\ignore{
\begin{proposition}\label{prop:condition_pfl}
	If it holds for all $i\in M$ and $j\in N$ that
	\begin{equation}\label{eq:condition_pfl}
		p\ge n - \sum_{k\in N}\frac{v_{ik}}{v_{min}^i}\left(1-\frac{w_k}{w_j}\right)-\sum_{k\in N}\frac{v_{ik}}{v_{ij}}\frac{w_k}{w_j},
	\end{equation}
	then the objective function of \eqref{eq:relaxedpfl} is a monotone nondecreasing and submodular function.
\end{proposition}
\begin{proof}
	According to Corollary~\ref{cor:monotonicity}, it suffices to ensure the following inequality holds for any $i\in M$ and $j\in N$
	\begin{align*}
		 &\frac{v_{ij}w_j}{v_{ij}-v_{min}^i}\ge\frac{\sum_{k\in N}v_{ik}w_k}{\sum_{k\in N}(v_{ik}-v_{min}^i)+pv_{min}^i}=\frac{\sum_{k\in N}v_{ik}w_k}{\sum_{k\in N}v_{ik}+(p-n)v_{min}^i},\\
		 \Leftrightarrow\;&(p-n)v_{min}^i+\sum_{k\in N}v_{ik}\ge\frac{(v_{ij}-v_{min}^i)\sum_{k\in N}v_{ik}w_k}{v_{ij}w_j}=\sum_{k\in N}v_{ik}\frac{w_k}{w_j}-v_{min}^i\sum_{k\in N}\frac{v_{ik}}{v_{ij}}\frac{w_k}{w_j},
	\end{align*}
	which isin equivalent to \eqref{eq:condition_pfl} by dividing $v^i_{min}$ and separating $p$ and the remainder.
\end{proof}
In an extreme case that all $v_{ik}$ are identical, it can be verified directly that the right hand side of \eqref{eq:condition_pfl} is 0, which implies submodularity holds anyway. It coincides with the fact that in this case, the objective of \eqref{eq:pfl} is a modular, i.e. linear, function. Hence, Proposition~\ref{prop:condition_pfl} generalizes this special case. Finally, Proposition~\ref{prop:condition_pfl} suggests that if the number of facilities to build is sufficiently large, then the greedy algorithm is worthwhile to utilize to get a constant approximation guarantee. This may happens, for example, when a chain retail plans to open a store in every large city of a state, but due to limited budget, a small fraction of the candidate cities has to be excluded or to be opened with the lowest priority.
}

\subsection{On minimization problems}

In this note we focus on identifying submodularity in maximization problems, in which case greedy algorithms can be used to obtain near optimal solutions. However, submodularity can be exploited in minimization problems as well. Indeed, the epigraph of a submodular set function is described by its Lov\'asz extension \citep{lovasz1983submodular}, which can be used to improve mixed-integer programming formulations via cutting planes. Moreover, even if a given ratio is not submodular, the results presented in this paper can be used to decompose any ratio into two components such that one of which is submodular (and strengthening can be done using the submodular component); see, e.g., \cite{atamturk2019submodular}.

\section{Conclusion}\label{sec:conclusions}
In this note we explore submodularity of the objective function for a broad class of fractional 0--1 programs with multiple-ratios. Under some mild assumptions, we derive the necessary and sufficient condition for a single ratio of two linear functions to be submodular. Therefore, if the derived condition holds for every considered single-ratio function, then simple greedy algorithms can be used to deliver good quality solutions for multiple-ratio fractional 0--1 programs. Finally, we also illustrate applicability of our results in the context of the assortment optimization and facility location problems.

\vspace{1mm}



{\textbf{Acknowledgements.}} This note is based upon work supported by the National Science Foundation
under Grant No.~1818700.
\bibliographystyle{apalike}
\bibliography{Bibliography}
\end{document}